\documentclass[11pt]{amsart}

\usepackage[square,numbers]{natbib}
\usepackage{enumerate,graphicx}
\usepackage{MnSymbol}
\usepackage{mathrsfs,bbm} 

\newtheorem{theorem}{Theorem}[section]

\newtheorem{proposition}[theorem]{Proposition}

\newtheorem{lemma}[theorem]{Lemma}

\theoremstyle{definition}

\newtheorem{example}[theorem]{Example}

\allowdisplaybreaks

\begin{document}

\title[Reversibility of LCA on Cayley Tree with PBC]{Reversibility of Linear Cellular Automata on Cayley Trees with Periodic Boundary Condition}



\author{Chih-Hung Chang}
\author{Jing-Yi Su}
\address{Department of Applied Mathematics, National University of Kaohsiung, Kaohsiung 81148, Taiwan, ROC.}
\email{chchang@nuk.edu.tw; snowinds127@gmail.com}
\baselineskip=1.2\baselineskip

\begin{abstract}
While one-dimensional cellular automata have been well studied, there are relatively few results about multidimensional cellular automata; the investigation of cellular automata defined on Cayley trees constitutes an intermediate class. This paper studies the reversibility of linear cellular automata defined on Cayley trees with periodic boundary condition, where the local rule is given by $f(x_0, x_1, \ldots, x_d) = b x_0 + c_1 x_1 + \cdots + c_d x_d \pmod{m}$ for some integers $m, d \geq 2$. The reversibility problem relates to solving a polynomial derived from a recurrence relation, and an explicit formula is revealed; as an example, the complete criteria of the reversibility of linear cellular automata defined on Cayley trees over $\mathbb{Z}_2$, $\mathbb{Z}_3$, and some other specific case are addressed. Further, this study achieves a possible approach for determining the reversibility of multidimensional cellular automata, which is known as a undecidable problem.
\end{abstract}

\maketitle

\section{Introduction}

Cellular automaton (CA) is a particular class of dynamical systems introduced by Ulam and von Neumann as a model for self-production and is widely studied in multidisciplinary areas such as physics, biology, image processing, cryptography, and pseudo-random number generation \cite{CCD-CMA1999, DC-IS2004, FCC-PRE1998, KCD+-CMA1999, YZ-IS2009}. One-dimensional CA consists of infinite lattice with finite states and a local rule; the celebrated works of Hedlund and Wolfram make a decisive impulse to the mathematical study of CA (see \cite{Hed-MST1969, Wol-2002} and the references therein).

A cellular automaton is called reversible if every current configuration is associated with exactly one past configuration. While the reversibility of one-dimensional CAs is elucidated \citep{ION-JCSS1983, Morit-2012, Nasu-TAMS2002}, Kari indicates that the reversibility of multidimensional CAs is undecidable \cite{Kari-PD1990, Kari-JCSS1994}. Recently, the reversibility problem for one-dimensional cellular automata with boundary conditions has been studied \cite{CAS-JSP2011, ER-AMC2007a}; it is shown that a reversibly linear CA is either a Bernoulli automorphism or non-ergodic \cite{CC-IS2016}. Many researchers have been devoted to investigating the reversibility problem of multidimensional linear cellular automata under boundary conditions; however, there is no algorithm for determining whether a multidimensional linear cellular automaton is reversible \cite{KSA-TJM2016, SAK-IJMPC2012, YWX-IS2015}.

The notion of CAs have been extended to the case where the underlying space is the Cayley tree of a finitely generated group or semigroup (see \cite{AB-TCS2013, BC-2015, CCF+-TCS2013} and the references therein). Note that the grid $\mathbb{Z}^d$ is the Cayley tree of the free abelian group with $d$ generators. In other words, CAs defined on Cayley trees constitute an intermediate class in between one-dimensional and multidimensional CAs.

Some interesting phenomena are observed in reversible CA defined on Cayley trees. Figure \ref{fig:p3} shows an eventually periodic orbit of a linear CA defined on binary Cayley tree of height $5$ over three symbols. Furthermore, every configuration in the orbit is symmetric.

\begin{figure}
\begin{center}
\includegraphics[scale=0.5]{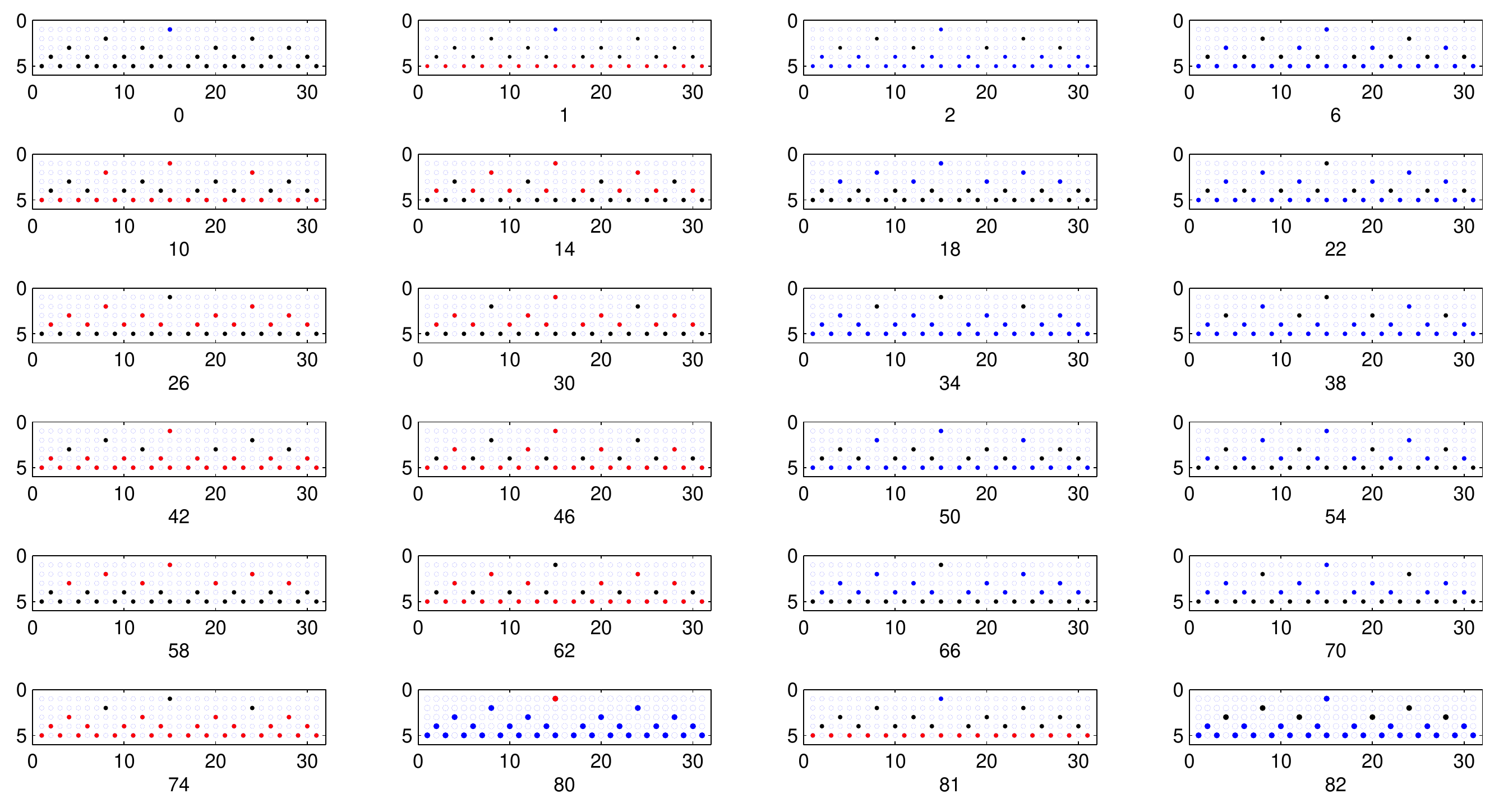}
\end{center}
\caption{An eventually periodic orbit of CA defined on binary tree with height $5$, where the local rule is given by $f(x_0, x_1, x_2) = x_0 + x_1 + x_2 \pmod{3}$. The number below each figure indicates the number of the time step; it is seen that this is an eventually periodic orbit with period $80$. The color black, blue, and red represents $0$, $1$, and $2$, respectively.}
\label{fig:p3}
\end{figure}

In this article, we investigate the reversibility of linear CAs defined on Cayley trees with periodic boundary condition; the matrix algebra has been used for the investigation. We define the matrix presentation of a linear CA defined on a Cayley tree of height $n$, and then characterize the reversibility by demonstrating an explicit formula of the determinant of its related matrix. The main difficulty comes from the exponential growth rate of the dimension of the matrix whenever $n$ increases. Furthermore, the complete criteria for the reversibility of linear CAs defined on Cayley trees over $\mathbb{Z}_2$, $\mathbb{Z}_3$, and some other condition are addressed.

The rest of this paper is organized as follows. The upcoming section introduces some preliminaries and the reversibility of linear CAs defined on binary Cayley tree. Section $3$ elaborates the explicit criteria for the reversibility of linear CAs in some cases. Section $4$ addresses the proofs of theorems stated in Section $2$ while Section $5$ extends all the results to general Cayley trees. Some discussion and conclusion are given in Section $6$.

\section{Reversibility of CA on Binary Tree}

This section investigates the reversibility of a cellular automaton defined on binary Cayley tree for the clarification; the discussion extends to cellular automata defined on general Cayley trees and is illustrated later.

We start with some basic definitions of symbolic dynamics on infinite binary trees. Let $\Sigma = \{1, 2\}$ and let $\Sigma^*$ be the set of words over $\Sigma$; more specifically, $\Sigma^* = \bigcup_{n \geq 0} \Sigma^n$, where $\Sigma^n = \{w_1 w_2 \cdots w_n: w_i \in \Sigma \text{ for } 1 \leq i \leq n\}$ is the set of words of length $n$ for $n \in \mathbb{N}$ and $\Sigma^0 = \{\epsilon\}$ consists of the empty word $\epsilon$. An \emph{infinite tree} $t$ over a finite alphabet $\mathcal{A}=\{0, 1, \ldots, m-1\}$, herein $m \geq 2$, is a function from $\Sigma^*$ to $\mathcal{A}$; a \emph{node} of an infinite tree is a word of $\Sigma^*$ while the empty word relates to the root of the tree. Suppose that $x$ is a node of a tree; $x$ has two \emph{children} $x1$ and $x2$, and $x$ is the \emph{parent} of $x1$ and $x2$. Furthermore, a node without children is called a \emph{leaf}.

Let $t$ be a tree and let $x$ be a node, we refer $t_x$ to $t(x)$ for simplicity. A subset of words $L \subset \Sigma^*$ is called \emph{prefix-closed} if each prefix of $L$ belongs to $L$. A function $u$ defined on a finite prefix-closed subset $L$ with codomain $\mathcal{A}$ is called a \emph{pattern} or \emph{block}, and $L$ is called the \emph{support} of the pattern. Suppose that $n$ is a nonnegative integer. Let $\Sigma_n = \bigcup_{k = 0}^n \Sigma^k$ denote the set of words of length at most $n$. We say that a pattern $u$ is \emph{a block of height $n$} (or \emph{an $n$-block}) if the support of $u$ is $\Sigma_{n-1}$, denoted by $\mathrm{height}(u) = n$.

Let $\mathcal{T} = \mathcal{A}^{\Sigma^*}$ and $\mathcal{T}_n = \mathcal{A}^{\Sigma_{n-1}}$ be the sets of infinite trees and $n$-blocks over $\mathcal{A}$, respectively, where $n \in \mathbb{N}$. $T_f: \mathcal{T} \to \mathcal{T}$ is called a cellular automaton defined on Cayley tree (TCA) $\mathcal{T}$ with local rule $f$ if there exists $k \in \mathbb{N}$, an ordered set $\mathcal{N} = \{y_1, \ldots, y_k\} \subset  \Sigma^*$, and a local map $f: \mathcal{A}^k \to \mathcal{A}$ such that $(T_f t)_x = f(t_{x \mathcal{N}})$ for all $x \in \Sigma^*$, where $x \mathcal{N} = \{xy_1, \ldots, xy_k\}$ and $f(t_{x \mathcal{N}}) = f(xy_1, \ldots, xy_k)$; $T_f$ is called linear if $f$ is linear.

This elucidation focuses on those linear TCA $T_f$ whose local rule $f: \mathcal{A}^3 \rightarrow \mathcal{A}$ is given by $f (x_0, x_1, x_2) = b x_0 + c_1 x_1 + c_2 x_2 \pmod{m}$; in other words, $T_f: \mathcal{T} \rightarrow \mathcal{T}$ is defined as
$$
(T_ft)_x = f(t_x,t_{x1},t_{x2}) = b t_x + c_1 t_{x1} + c_2 t_{x2}  \pmod{m}
$$
for $x\in \Sigma^*$. A TCA with periodic boundary condition (PBC) is a cellular automaton defined on a finite tree such that the ``children'' of leaves is the root. More explicitly, a linear TCA with PBC is $T_f: \mathcal{T}_n \to \mathcal{T}_n$ defined as
\begin{equation}
(T_ft)_x =
\left\{
\begin{array}{ll}
f(t_x,t_{x1},t_{x2}), & |x|\leq n-2; \\
f(t_x,t_{\epsilon},t_{\epsilon}), & |x|=n-1,
\end{array}
\right.
\end{equation}
for some $n \geq 2$, where $|x|$ indicates the length of $x$.

To investigate the reversibility of linear TCA with PBC, we transfer the global map $T_f$ into matrix operation. Let $\beta$ be an ordered set obtained by rearranging $\Sigma_{n-1}$ with respect to the lexicographical order. For each $t \in \mathcal{T}_n$, denote by $[t]_{\beta} \in \mathcal{A}^{2^{n}-1}$ the column vector of $t$ with respect to $\beta$; more precisely, let $\Theta: \Sigma_{n-1} \rightarrow \mathbb{N}$ be defined as
\begin{equation}
\Theta(x)=
\left\{
\begin{array}{ll}
1, & x=\epsilon; \\
1+\sum^{k}_{i=1} x_{i} 2^{k-i}, & x = x_1 \cdots x_k \in \Sigma^{k};
\end{array}
\right.
\end{equation}
then $[t]_{\beta}(i)=t_x$ for $x \in \Sigma_n$ with $\Theta(x)=i$.
Let
\begin{equation}
T_n (i, j) =
\left\{
\begin{array}{ll}
b, & j=i, 1\leq i \leq 2^{n}-1; \\
c_1, & j=2i, 1\leq i\leq 2^{n-1}-1; \\
c_2, & j=2i+1, 1\leq i\leq 2^{n-1}-1; \\
c_1+c_2, & j=1, 2^{n-1}\leq i\leq 2^{n}-1; \\
0, & \mbox{otherwise.}
\end{array}
\right.
\end{equation}
Evidently, $[T_f t]_{\beta} = T_n [t]_{\beta} \pmod{m}$, and we call $T_n$ the \emph{matrix presentation} of the linear TCA $T_f$ defined on $\mathcal{T}_n$ with PBC.

\begin{example}
Consider the case where $n = 3$, it can be verified without difficulty that
\[
T_3 =
\left(
\begin{array}{ccccccc}
  b & c_1 & c_2 & 0 & 0 & 0 & 0 \\
  0 & b & 0 & c_1 & c_2 & 0 & 0 \\
  0 & 0 & b & 0 & 0 & c_1 & c_2 \\
c_1+c_2 & 0 & 0 & b & 0 & 0 & 0 \\
c_1+c_2 & 0 & 0 & 0 & b & 0 & 0 \\
c_1+c_2 & 0 & 0 & 0 & 0 & b & 0 \\
c_1+c_2 & 0 & 0 & 0 & 0 & 0 & b
\end{array}
\right)
\]
and
\[
[T_ft]_{\beta} =
\begin{pmatrix}
  b & c_1 & c_2 & 0 & 0 & 0 & 0 \\
  0 & b & 0 & c_1 & c_2 & 0 & 0 \\
  0 & 0 & b & 0 & 0 & c_1 & c_2 \\
c_1+c_2 & 0 & 0 & b & 0 & 0 & 0 \\
c_1+c_2 & 0 & 0 & 0 & b & 0 & 0 \\
c_1+c_2 & 0 & 0 & 0 & 0 & b & 0 \\
c_1+c_2 & 0 & 0 & 0 & 0 & 0 & b
\end{pmatrix}
\begin{pmatrix}
t_{\epsilon} \\
t_1 \\
t_2 \\
t_{11} \\
t_{12} \\
t_{21} \\
t_{22}
\end{pmatrix}
\]
for all $t \in \mathcal{T}_3$.
\end{example}

The following proposition follows immediately from the definition of $[\cdot]_{\beta}$, and the proof is omitted.

\begin{proposition}\label{prop:reversible-matrix-determinant}
The following are equivalent.
\begin{enumerate}[\bf (i)]
  \item TCA $T_f$ with PBC is reversible.
  \item $T_n$ is invertible over $\mathbb{Z}_m$.
  \item $\det T_n \neq 0 \pmod{m}$.
\end{enumerate}
\end{proposition}

Proposition \ref{prop:reversible-matrix-determinant} indicates that, to determine whether $T_f$ is reversible, it is equivalent to elucidate if $\det T_n \neq 0 \pmod{m}$. Without loss of generality, we may assume that $b \neq 0$; Theorem \ref{thm:formula-Tn} obtains the explicit formula of $\det T_n$.

\begin{theorem}\label{thm:formula-Tn}
Let $T_n$ be the matrix presentation of the linear TCA $T_f$ defined on $\mathcal{T}_n$ with PBC. Then
\begin{equation}
\det T_n=(-1)^{2^n-n-1}(b+c_1+c_2)b^{2^n-n-1}\left[\sum^{n-1}_{j=0}(-1)^{j}(c_1+c_2)^{n-1-j}b^{j}\right]
\end{equation}
\end{theorem}

The proof of Theorem \ref{thm:formula-Tn} is postponed to Section \ref{sec:proof-main-theorem}. The following example reveals the main idea of the proof.

\begin{example}
For the case where $n = 4$, it is seen that
\begin{align*}
\det T_4 &= \det
\left(
  \begin{array}{ccccccccccccccc}
      b & c_1 & c_2 &   0 &   0 &   0 &   0 & 0 & 0 & 0& 0& 0& 0& 0& 0\\
      0 &   b &   0 & c_1 & c_2 &   0 &   0 & 0 & 0 & 0& 0& 0& 0& 0& 0\\
      0 &   0 &   b &   0 &   0 & c_1 & c_2 & 0 & 0 & 0& 0& 0& 0& 0& 0\\
0 &   0 &   0 &   b &   0 &   0 &   0 & c_1 & c_2 & 0& 0& 0& 0& 0& 0\\
0 &   0 &   0 &   0 &   b &   0 &   0 & 0 & 0& c_1& c_2& 0& 0& 0& 0\\
0 &   0 &   0 &   0 &   0 &   b &   0 & 0 & 0 & 0& 0& c_1& c_2& 0& 0\\
0 &   0 &   0 &   0 &   0 &   0 &   b & 0 & 0 & 0& 0& 0& 0& c_1& c_2\\
c_1+c_2 &   0 &   0 &   0 &   0 &   0 &   0 & b & 0 & 0& 0& 0& 0& 0& 0\\
c_1+c_2 &   0 &   0 &   0 &   0 &   0 &   0 & 0 & b & 0& 0& 0& 0& 0& 0\\
c_1+c_2 &   0 &   0 &   0 &   0 &   0 &   0 & 0 & 0 & b& 0& 0& 0& 0& 0\\
c_1+c_2 &   0 &   0 &   0 &   0 &   0 &   0 & 0 & 0 & 0& b& 0& 0& 0& 0\\
c_1+c_2 &   0 &   0 &   0 &   0 &   0 &   0 & 0 & 0 & 0& 0& b& 0& 0& 0\\
c_1+c_2 &   0 &   0 &   0 &   0 &   0 &   0 & 0 & 0 & 0& 0& 0& b& 0& 0\\
c_1+c_2 &   0 &   0 &   0 &   0 &   0 &   0 & 0 & 0 & 0& 0& 0& 0& b& 0\\
c_1+c_2 &   0 &   0 &   0 &   0 &   0 &   0 & 0 & 0 & 0& 0& 0& 0& 0& b\\
  \end{array}
\right) \\
&= (b+c_1+c_2)\det
\left(
  \begin{array}{ccccccccccccccc}
0 & c_1 & c_2 &   0 &   0 &   0 &   0 &   -b & 0 & 0& 0& 0& 0& 0& 0\\
0 &   b &   0 & c_1 & c_2 &   0 &   0 &   -b & 0 & 0& 0& 0& 0& 0& 0\\
0 &   0 &   b &   0 &   0 & c_1 & c_2 &   -b & 0 & 0& 0& 0& 0& 0& 0\\
0 &   0 &   0 &   b &   0 &   0 &   0 & c_1-b & c_2 & 0& 0& 0& 0& 0& 0\\
0 &   0 &   0 &   0 &   b &   0 &   0 &   -b & 0& c_1& c_2& 0& 0& 0& 0\\
0 &   0 &   0 &   0 &   0 &   b &   0 &   -b & 0 & 0& 0& c_1& c_2& 0& 0\\
0 &   0 &   0 &   0 &   0 &   0 &   b &   -b & 0 & 0& 0& 0& 0& c_1& c_2\\
1 &   0 &   0 &   0 &   0 &   0 &   0 &   b & 0 & 0& 0& 0& 0& 0& 0\\
0 &   0 &   0 &   0 &   0 &   0 &   0 &   -b & b & 0& 0& 0& 0& 0& 0\\
0 &   0 &   0 &   0 &   0 &   0 &   0 &   -b & 0 & b& 0& 0& 0& 0& 0\\
0 &   0 &   0 &   0 &   0 &   0 &   0 &   -b & 0 & 0& b& 0& 0& 0& 0\\
0 &   0 &   0 &   0 &   0 &   0 &   0 &   -b & 0 & 0& 0& b& 0& 0& 0\\
0 &   0 &   0 &   0 &   0 &   0 &   0 &   -b & 0 & 0& 0& 0& b& 0& 0\\
0 &   0 &   0 &   0 &   0 &   0 &   0 &   -b & 0 & 0& 0& 0& 0& b& 0\\
0 &   0 &   0 &   0 &   0 &   0 &   0 &   -b & 0 & 0& 0& 0& 0& 0& b\\
  \end{array}
\right) \\
 &= -(b+c_1+c_2)b^{7}\det
\left(
  \begin{array}{cccccccc}
  c_1 & c_2 &   0 &   0 &   0 &   0 & -b \\
    b &   0 & c_1 & c_2 &   0 &   0 & -b \\
    0 &   b &   0 &   0 & c_1 & c_2 & -b \\
    0 &   0 &   b &   0 &   0 &   0 & c_1+c_2-b \\
    0 &   0 &   0 &   b &   0 &   0 & c_1+c_2-b \\
    0 &   0 &   0 &   0 &   b &   0 & c_1+c_2-b \\
    0 &   0 &   0 &   0 &   0 &   b & c_1+c_2-b
  \end{array}
\right) \\
 &= (b+c_1+c_2)(-b)^{7}\det
\left(
  \begin{array}{cccccccc}
  c_1 & c_2 &   0 &   0 &   0 &   0 & -b \\
    b &   0 &   0 &   0 &   0 &   0 & -\dfrac{(c_1 + c_2)(c_1 + c_2 - b) + b^2}{b} \\
    0 &   b &   0 &   0 &   0 &   0 & -\dfrac{(c_1 + c_2)(c_1 + c_2 - b) + b^2}{b} \\
    0 &   0 &   b &   0 &   0 &   0 & c_1+c_2-b \\
    0 &   0 &   0 &   b &   0 &   0 & c_1+c_2-b \\
    0 &   0 &   0 &   0 &   b &   0 & c_1+c_2-b \\
    0 &   0 &   0 &   0 &   0 &   b & c_1+c_2-b
  \end{array}
\right) \\
&= (b+c_1+c_2)(-b)^{11}\det
\left(
  \begin{array}{ccc}
 c_1 & c_2 & -b \\
  b &   0 & -\dfrac{(c_1 + c_2)(c_1 + c_2 - b) + b^2}{b} \\
  0 &   b & -\dfrac{(c_1 + c_2)(c_1 + c_2 - b) + b^2}{b}
  \end{array}
\right)\\
&= - (b+c_1+c_2)b^{11}\left(\sum_{j=0}^3 (-1)^j (c_1 + c_2)^{3-j} b^j\right)
\end{align*}
\end{example}

\section{Necessary and Sufficient Condition for the Reversibility}

This section addresses the full criteria for the reversibility of linear TCA over $\mathbb{Z}_2$, $\mathbb{Z}_3$, and some other case. The criteria herein focus on the binary tree for the simplicity and can extend to general Cayley trees without difficulty.

Proposition \ref{prop:reversible-matrix-determinant} and Theorem \ref{thm:formula-Tn} reveal that the necessary and sufficient condition for a TCA $T_f$ being reversible is
$$
(b+c_1+c_2)b^{2^n-n-1}\left[\sum^{n-1}_{j=0}(-1)^{j}(c_1+c_2)^{n-1-j}b^{j}\right] \neq 0 \pmod{m}.
$$
Suppose that $m=p$ is a prime. The necessary and sufficient condition for the reversibility of $T_f$ can be simplified as follows: $T_f$ is reversible if and only if
\begin{enumerate}[\bf (i)]
\item $b \neq 0 \pmod{p}$;
\item $b + c_1 + c_2 \neq 0 \pmod{p}$;
\item $\sum^{n-1}_{j=0}(-1)^{j}(c_1+c_2)^{n-1-j}b^{j} \neq 0 \pmod{p}$.
\end{enumerate}
To ease the notation, all the evaluation are considered under modulo $p$ in this section unless otherwise stated.

Let $g(x) = \sum^{n-1}_{k=0} (-1)^{k} b^{k} x^{n-1-k}$; elaborating the reversibility of $T_f$ is ``almost'' equivalent to the discussion of the existence of the roots of $g(x)$. This section investigates the cases where $p = 2$ and $p = 3$; for each case, the reversibility of $T_f$ is characterized explicitly. For this purpose, we assume that $b \neq 0$ in the rest of this section.

Furthermore, the reversibility of $T_f$ for the case where $n = 2^l$ for some $l \in \mathbb{N}$ is elucidated; some interesting algebraic property is revealed.

\subsection{Cases Study: $p=2$}

\begin{proposition}
$T_f$ is reversible if and only if $b=1$ and $c_1+c_2=0$.
\end{proposition}
\begin{proof}
Notably, we only need to consider the case where $b = 1$ when $p = 2$ since $b \neq 0$. It is seen that
$$
g(x) = 1 + x + x^2 + \cdots + x^{n-1}.
$$
Suppose that $n$ is even. It is easily seen that $g(x) = 0$ if and only if $x = 1$. That is, $T_f$ is reversible if and only if $c_1 + c_2 = 0$. For the case where $n$ is odd, $g(x)$ can be expressed as
$$
g(x) = (1+x)(1 + x^2 + x^4 + \cdots + x^{n - 2}) = 0 \pmod{2} \quad \Leftrightarrow \quad x = 1
$$
Both cases demonstrate that $T_f$ is reversible if and only if $c_1 + c_2 = 0$. This completes the proof.
\end{proof}

Figure \ref{fig:p2} illustrates a periodic orbit derived by a reversible CA defined on binary tree with height $5$.
\begin{figure}
\begin{center}
\includegraphics[scale=0.5]{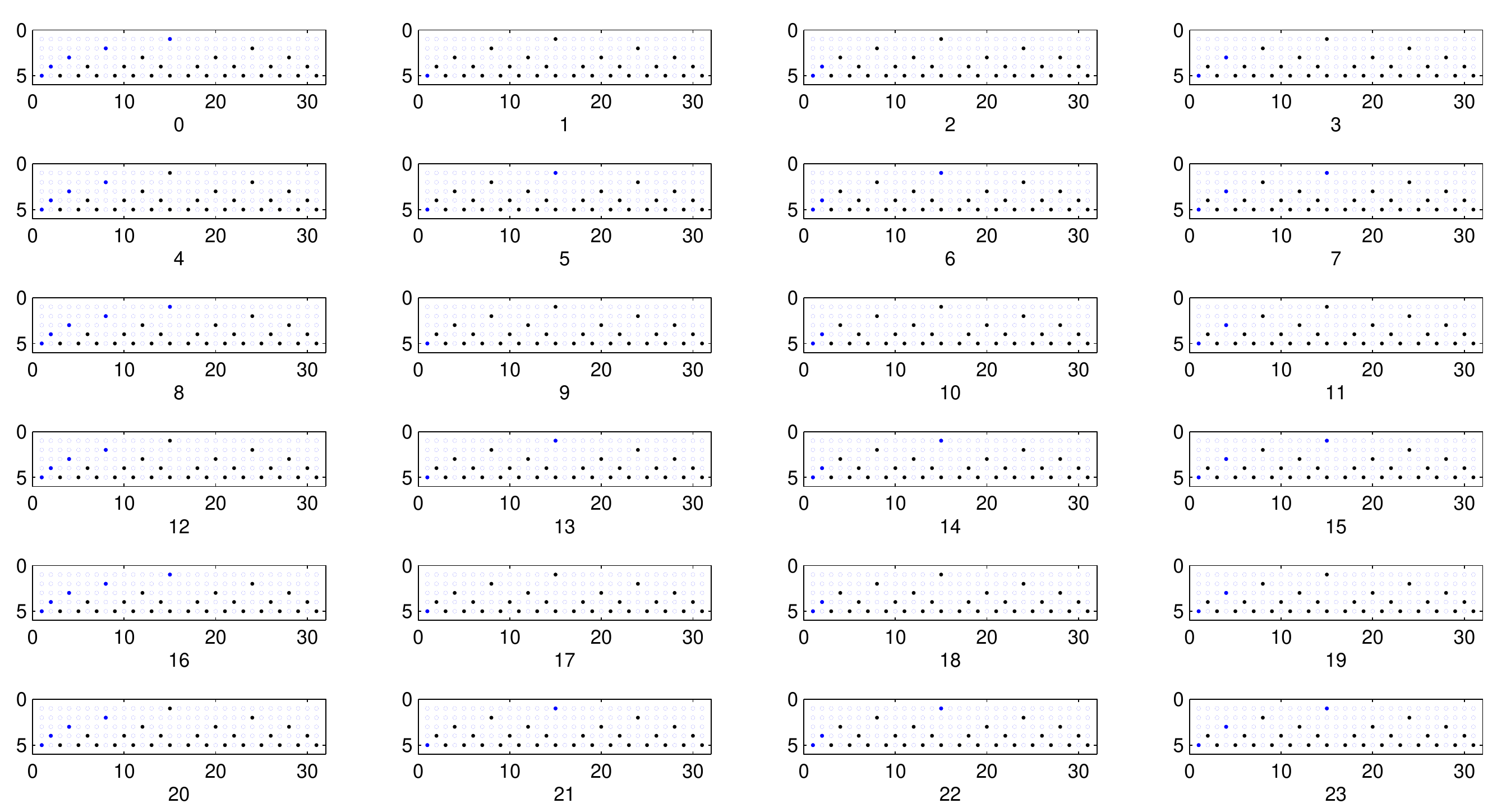}
\end{center}
\caption{A period-$8$ orbit of CA defined on binary tree with height $5$, where the local rule is given by $f(x_0, x_1, x_2) = x_0 + x_1 + x_2 \pmod{2}$. The color black and blue represents $0$ and $1$, respectively.}
\label{fig:p2}
\end{figure}

\subsection{Cases Study: $p=3$}

Recall that $b = 0$ indicates that $T_f$ is irreversible; hence it suffices to consider the case where $b \neq 0$.

\begin{proposition}
For the case where $p=3$, $T_f$ is irreversible if and only if one of the following is satisfied:
\begin{enumerate}[\bf 1)]
\item $c_1+c_2=b$;
\item $c_1+c_2\neq 0$ and $n$ is a multiple of $6$;
\item $b+c_1+c_2=0$.
\end{enumerate}
\end{proposition}

\begin{proof}
Since $b \neq 0$, Fermat's Little Theorem asserts that $b^2=1 \pmod{3}$.

\noindent \textbf{Case 1.} $n$ is even.
\begin{align*}
g(x)&=x^{n-1}-bx^{n-2}+b^2x^{n-3}-b^3x^{n-4}+...+b^{n-2}x-b^{n-1}\\
&=(x-b)(x^{n-2}+x^{n-4}+x^{n-6}+...+x^2+1) \pmod{3}
\end{align*}
A straightforward examination shows that
$$
x^{n-2}+x^{n-4}+x^{n-6}+...+x^2+1 = 0 \pmod{3} \quad \Leftrightarrow \quad x \neq 0 \text{ and } n=6l
$$
for some $l \in \mathbb{N}$. Therefore, $g(x) = 0 \pmod{3}$ if and only if $x=b$ or $x \neq 0$ and $n=6l$  for some $l \in \mathbb{N}$.

\noindent \textbf{Case 2.} $n$ is odd. Observe that
\begin{align*}
g(x) &= \left\{
     \begin{array}{ll}
       (x-1)(x^{n-2} + x^{n-4} + \cdots + x^2 + 1), & b=1\hbox{;} \\
       (x+1)(x^{n-2} + x^{n-4} + \cdots + x^2 + 1), & b = 2\hbox{.}
     \end{array}
   \right. \pmod{3}
\end{align*}
Similar to the above discussion, it follows that $g(x) = 0 \pmod{3}$ if and only if $x=b$ or $x \neq 0$ and $n=6l$  for some $l \in \mathbb{N}$.

The proof is complete.
\end{proof}

\subsection{Cases Study: $n=2^l$}

Instead of considering the special cases $p = 2$ and $p = 3$, this subsection focuses on the case where the height of the Cayley tree is some power of $2$. For this case, some algebraic property related to the reversibility of $T_f$ is revealed.

Suppose that $n = 2^l$ for some $l \in \mathbb{N}$. It is seen that
\begin{align*}
g(x)&=(x-b)(x^{n-2}+b^2x^{x-4}+b^4x^{n-6}+...+b^{n-4}x^2+b^{n-2}) \\
&=(x-b)(x^2+b^2)(x^4+b^4) \cdots (x^{2^{l-1}}+b^{2^{l-1}})
\end{align*}
The last equality can be demonstrated via the mathematical induction principle, thus is omitted. The reversibility of $T_f$ is then characterized explicitly as follows.


\begin{proposition}
For the case where $n=2^{l}$ for some $l \in \mathbb{N}$, $T_f$ is irreversible if and only if one of the following is satisfied:
\begin{enumerate}[\bf 1)]
\item $b + c_1 + c_2 = 0$;
\item $c_1+c_2=b$;
\item $p=2^{r}+1$ for some integer $r < l$ and $(c_1 + c_2)^{2^{q}}+b^{2^{q}}=0 \pmod{p}$ for some $1 \leq q \leq r-1$;
\item $p \neq 2^r + 1$ for $r \geq 0$ and $(c_1 + c_2)^{2^{q}}+b^{2^{q}}=0 \pmod{p}$ for some $1 \leq q \leq \gamma$, where $\gamma$ is the order of the subset $A = \{2^k: k \in \mathbb{N}\}$ of $\mathbb{Z}_p$.
\end{enumerate}
\end{proposition}
\begin{proof}
Observe that $g(x) \neq 0$ if $x = 0$ and $T_f$ is reversible in this case; therefore, we may assume that $x \neq 0$ without loss of generality. It is obvious that the irreversibility of $T_f$ follows from $c_1 + c_2 = b$.

Suppose that $p=2^{r}+1$ for some $r \in \mathbb{N}$. Fermat's Little Theorem asserts that $a^{2^k} = 1 \pmod{p}$ for all $k \geq r$ provided $a \neq 0$. It follows immediately that
$$
g(x)=2^{l-r}(x-b) (x^2+b^2) (x^4+b^4) \cdots (x^{2^{r-1}}+b^{2^{r-1}}).
$$
Hence, $g(x) = 0$ if and only if $x = b$ or $x^{2^{q}}+b^{2^{q}}=0 \pmod{p}$ for some $1 \leq q \leq r-1$.

Suppose that the prime $p$ cannot expressed as $p = 2^r + 1$ for $r \geq 0$. It can be verified that the sequence $\{2^k\}_{k \in \mathbb{N}}$ is eventually periodic. Let $A = \{2^k: k \geq 1\}$ be an ordered set with $|A| = \gamma$. There exists $\lambda \in \mathbb{N}$ such that
$$
g(x) = (x - b) \left(\prod_{i=1}^{\lambda} (x^{2^i} + b^{2^i})\right) \left(\prod_{i=\lambda+1}^{\tau} (x^{2^i} + b^{2^i})\right)^{\kappa+1} \left(\prod_{i=\tau+1}^{\gamma} (x^{2^i} + b^{2^i})\right)^{\kappa}
$$
for some $\kappa$ and $\tau$. Evidently, $g(x) = 0$ if and only if $x = b$ or $x^{2^{q}}+b^{2^{q}}=0$ for some $1 \leq q \leq \gamma$.

This completes the proof.
\end{proof}

We end this section with the following two examples that illustrate the case where $p \neq 2^r +1$ and $n = 2^l$.

\begin{example}
Suppose that $n = 2^{10}$ and $p = 7$. A routine examination infers that $x^{2^{k}}=x^2 \pmod{7}$ if $k$ is odd and $x^{2^{k}}=x^4 \pmod{7}$ if $k$ is even; that is, $A = \{2, 4\}$. It follows that
$$
g(x) = (x-b)(x^2+b^2)(x^4+b^4) \cdots (x^{2^9}+b^{2^9}) = (x-b)(x^2+b^2)^{5}(x^4+b^4)^{4}.
$$
Therefore, $T_f$ is reversible if and only if
\begin{enumerate}[\bf 1)]
\item $b+c_1+c_2 \neq 0$;
\item $c_1+c_2 \neq b$ and $b \neq 0$;
\item $(c_1+c_2)^2+b^2 \neq 0$;
\item $(c_1+c_2)^4+b^4 \neq 0$.
\end{enumerate}
\end{example}

\begin{example}
Suppose that $n = 2^{10}$ and $p = 29$. A routine examination infers that $x^{2^{k}}=x^4 \pmod{29}$ when $k=3n-1$; $x^{2^{k}}=x^8 \pmod{29}$ when $k=3n$, and $x^{2^{k}}=x^{16} \pmod{29}$ when $k=3n+1$  for some natural number $n \in \mathbb{N}$. In other words,
$$
\{2^k: k \geq 1\} = \{2, 4, 8, 16, 4, 8, 16, \ldots\} \quad \text{and} \quad A = \{2, 4, 8, 16\}.
$$
It follows that
$$
g(x) = (x-b)(x^2+b^2)(x^4+b^4) \cdots (x^{2^9}+b^{2^9}) = (x-b)(x^2+b^2)(x^4+b^4)^{3}(x^8+b^8)^{3}(x^{16}+b^{16})^{2}.
$$
Therefore, $T_f$ is reversible if and only if
\begin{enumerate}[\bf 1)]
\item $b+c_1+c_2 \neq 0$;
\item $c_1+c_2 \neq b$ and $b \neq 0$;
\item $(c_1+c_2)^2+b^2 \neq 0$;
\item $(c_1+c_2)^4+b^4 \neq 0$;
\item $(c_1+c_2)^8+b^8 \neq 0$;
\item $(c_1+c_2)^{16}+b^{16} \neq 0$.
\end{enumerate}

\end{example}

\section{Proof of Theorem \ref{thm:formula-Tn}} \label{sec:proof-main-theorem}

This section contributes to the proof of Theorem \ref{thm:formula-Tn}; we begin with the following lemma.

\begin{lemma} \label{combine21}
For $n\in \mathbb N$, let $\overline{T}_{n} \in \mathcal{M}_{2^{n-1}-1}(\mathbb{R})$ be a $(2^{n-1}-1) \times (2^{n-1}-1)$ matrix defined as
\begin{equation}
\overline{T}_{n}(i,j)=
\left\{
\begin{array}{ll}
       -b, & j=2^{n-1}-1,1\leq{i}\leq{2^{n-2}-1}; \\
c_1+c_2-b, & j=2^{n-1}-1,2^{n-2}\leq{i}\leq{2^{n-1}-1}; \\
        b, & j=i-1,2\leq i\leq 2^{n-1}-1; \\
      c_1, & j=2i-1,1\leq i\leq2^{n-2}-1; \\
      c_2, & j=2i,1\leq i\leq2^{n-2}-1; \\
        0, & \hbox{otherwise.}
\end{array}
\right.
\end{equation}
Then $\det T_n = (-b)^{2^{n-1}-1}(b+c_1+c_2)\det\overline{T}_{n}$.
\end{lemma}
\begin{proof}
Adding the $k$th column of $T_n$ to the first column recursively for $k \geq 2$ derives a new matrix $T_{n}^{(1)}$ defined as
\begin{equation}
T_n^{(1)}(i,j)=
\left\{
\begin{array}{ll}
b, & j=i, 2\leq i\leq 2^{n}-1; \\
c_1, & j=2i, 1\leq i\leq 2^{n-1}-1; \\
c_2, & j=2i+1, 1\leq i\leq 2^{n-1}-1; \\
b+c_1+c_2, & j=1, 1\leq i\leq 2^{n}-1; \\
0, & \mbox{otherwise.}
\end{array}
\right.
\end{equation}
It is seen that $\det T_{n} = \det T_{n}^{(1)}$. Since every entry in the first column of $T_n^{(1)}$ is identical, substituting the entry of the first column by $1$ and then adding $-1$ times of the $2^{n-1}$th row to the other rows recursively produces a new matrix $T_n^{(2)}$ defined as
\begin{equation}
T_n^{(2)}(i,j)=
\left\{
\begin{array}{ll}
b, & j=i, 2\leq i\leq 2^{n}-1; \\
c_1, & j=2i, 1\leq i\leq 2^{n-1}-1, i\neq{2^{n-2}}; \\
c_2, & j=2i+1, 1\leq i\leq 2^{n-1}-1; \\
-b, & j=2^{n-1},1\leq i\leq 2^{n}-1 , i\neq{2^{n-1}}, i\neq{2^{n-2}}; \\
c_1-b, &i=2^{n-2}, j=2^{n-1}; \\
1, & j=1, i=2^{n-1}; \\
0, & \mbox{otherwise.}
\end{array}
\right.
\end{equation}
Evidently, $\det T_{n}^{(1)} = (b+c_1+c_2) \det T_{n}^{(2)}$.

After adding the $k$th column to the $2^{n-1}$th column of $T_n^{(2)}$ recursively for $k > 2^{n-1}$, it follows immediately that the only nonzero element of $k$th row of $T_n^{(2)}$ is $b$ for $k > 2^{n-1}$. Furthermore, it can be verified without difficulty that $\overline{T}_n$ is driven from the following two steps: (i) delete the $k$th row and $k$th column of $T_n^{(2)}$ for $k > 2^{n-1}$; (ii) delete the first column and the last row. This indicates that $\det T_n^{(2)} = (-b)^{2^{n-1}-1} \det \overline{T}_n$. Hence, we derive
$$
\det T_{n} = \det T_{n}^{(1)} = (b+c_1++c_2) \det T_{n}^{(2)} = (b+c_1+c_2) (-b)^{2^{n-1}-1} \det \overline{T}_{n}
$$
and the proof is complete.
\end{proof}

To evaluate $\det \overline{T}_n$, we need two more lemmas.

\begin{lemma} \label{alpha2}
Let $\alpha_3=c_1+c_2-b$ and let $\alpha_n$ be defined recursively as
$$
\alpha_n = -\dfrac{\alpha_{n-1}(c_1+c_2)+b^2}{b}, \quad n \geq 4.
$$
Then
\begin{equation}\label{eq:alpha-n}
\alpha_n=(-1)^{n-3}\dfrac{\sum^{n-2}_{j=0}(-1)^{j}(c_1+c_2)^{n-2-j}b^{j}}{b^{n-3}},
\end{equation}
where $n \geq 3$.
\end{lemma}
\begin{proof}
We prove it by induction. Let $K=c_1+c_2$; it is seen that \eqref{eq:alpha-n} holds for $n=3$. Suppose \eqref{eq:alpha-n} is true when $n=m$ for some $m \geq 3$. Then
\begin{align*}
\alpha_{m+1}&=-\dfrac{\alpha_mK+b^2}{b}\\
&=-\dfrac{\left[(-1)^{m-3}\dfrac{\sum^{m-2}_{j=0}(-1)^{j}K^{m-2-j}b^{j}}{b^{m-3}}\right]K+b^2}{b}\\
&=-\dfrac{(-1)^{m-3}(\sum^{m-2}_{j=0}(-1)^{j}K^{m-2-j}b^{j})K+b^{m-1}}{b^{m-2}}\\
&=(-1)^{m-2}\dfrac{(\sum^{m-2}_{j=0}(-1)^{j}K^{m-1-j}b^{j})+(-1)^{m-3} b^{m-1}}{b^{m-2}}\\
&=(-1)^{m-2}\dfrac{\sum^{m-1}_{j=0}(-1)^{j}K^{m-1-j}b^{j}}{b^{m-2}}
\end{align*}
That is, \eqref{eq:alpha-n} remains to be true for $n = m + 1$. The proof is then complete by the mathematical induction principle.
\end{proof}

\begin{lemma} \label{combine22}
Suppose that $0 \leq k \leq n-3$. Let
\begin{equation}
F_n^{k}(i,j)=
\left\{
\begin{array}{ll}
\alpha_{k+3}, & i=2^{n-k-1}-1,j\geq{2^{n-k-2}}; \\
\overline{T}_{n-k}(i, j), & \hbox{otherwise}.
\end{array}
\right.
\end{equation}
Then $\det \overline{T}_{n} = (-b)^{2^{n-1} - 4} \det F^{n-3}_{n}$, where $\alpha_k$ is the same as defined in Lemma \ref{alpha2}.
\end{lemma}
\begin{proof}
Observe that $F_n^0 = \overline{T}_n$. For $1 \leq i \leq 2^{n-2}-2$, adding $\dfrac{-c_2}{b}$ times of the $(2^{n-1}-(2i-1))$th row and $\dfrac{-c_1}{b}$ times of the $(2^{n-1}-2i)$th row to the $(2^{n-2}-i)$th row of $\overline{T}_{n}$ is followed by
$$
\overline{T}'_{n}(i,j) = \left\{
                          \begin{array}{ll}
                            0, & 2 \leq i \leq 2^{n-2}-1, 2i-1 \leq j \leq 2i\hbox{;} \\
                            \alpha_4, & j=2^{n-1}-1, 2 \leq i \leq 2^{n-2}-1; \\
                            \overline{T}_n(i, j), & \hbox{otherwise.}
                          \end{array}
                        \right.
$$
It is seen that there is only one nonzero entry, which is $b$, in the $k$th row of $\overline{T}'_{n}$ for $3 \leq k \leq 2^{n-1}-1$. Evidently, we derive that
$$
\det \overline{T}_{n} = \det \overline{T}'_{n} = (-b)^{2^{n-2}} \det F^{1}_{n}.
$$
Applying similar operation to $F^{1}_{n}$ infers that $\det F^{1}_{n} = (-b)^{2^{n-3}} \det F^{2}_{n}$; the desired result follows by repeating the procedure for $n-3$ times.
\end{proof}

Now we are on the stage of presenting the proof of Theorem \ref{thm:formula-Tn}.

\begin{proof}[Proof of Theorem \ref{thm:formula-Tn}]
Observe that
\begin{align*}
\det F^{n-3}_{n} &= \det
\left(
\begin{array}{ccc}
c_1 & c_2 & -b \\
b   & 0   & \alpha_n \\
0   & b   & \alpha_n \\
\end{array}
\right)\\
&=-bK\alpha_n-b^3\\
&=-bK(-1)^{n-3}\dfrac{\sum^{n-2}_{j=0}(-1)^{j}K^{n-2-j} b^{j}}{b^{n-3}}-b^3\\
&=(-1)^{n-2}\dfrac{\sum^{n-2}_{j=0}(-1)^{j} K^{n-1-j} b^{j} + (-1)^{n-1} b^{n-1}}{b^{n-4}}\\
&=(-1)^{n-2} b^{4-n} \sum^{n-1}_{j=0}(-1)^{j} K^{n-1-j} b^{j},
\end{align*}
where $K = c_1 + c_2$.

By integrating Lemmas \ref{combine21}, \ref{alpha2}, and \ref{combine22}, we demonstrate that
\begin{align*}
\det T_n &= (b+c_1+c_2)(-b)^{2^{n-1}-1} \det \overline{T}_n \\
&=(b+c_1+c_2)(-b)^{2^{n-1}-1}(-b)^{2^{n-1} - 4} \det F^{n-3}_{n}\\
&=(b+c_1+c_2)(-b)^{2^n - 5} (-1)^{n-2} b^{4-n} \sum^{n-1}_{j=0} (-1)^{j} K^{n-1-j} b^{j} \\
&=(b+c_1+c_2)(-b)^{2^n - n - 1} \sum^{n-1}_{j=0} (-1)^{j} K^{n-1-j} b^{j} \\
&=(b+c_1+c_2)(-b)^{2^n - n - 1} \sum^{n-1}_{j=0} (-1)^{j}(c_1+c_2)^{n-1-j}b^{j}
\end{align*}
and this completes the proof.
\end{proof}

\section{Reversibility of CA on General Cayley Trees}

This section extends results of cellular automata defined on binary tree to general Cayley trees; in other words, this section illustrates the case where $\Sigma = \{1, 2, \ldots, d\}$ for some $d \geq 2$. Let $d_n=d+d^2+...+d^{n}$; the matrix presentation $T_n$ of CA $T_f$, whose local rule is given by $f(x_0, x_1, \ldots, x_d) = b x_0 + \sum\limits_{i=1}\limits^d c_i x_i \pmod{m}$, defined on $\mathcal{T}_n$ is defined as
\begin{equation}
T_n(i,j)=
\left\{
\begin{array}{ll}
b, & i=j; \\
c_k, & j=(i-1)d+k+1, 1 \leq i \leq 1+d_{n-2}, 1 \leq k \leq d; \\
c_1+c_2+...+c_d, & j=1, 2+d_{n-2}\leq i\leq 1+d_{n-1}; \\
0, & \hbox{otherwise.}
\end{array}
\right.
\end{equation}
The reversibility of $T_f$ is related to the determinant of $T_n$ being nonzero. The following theorem extends Theorem \ref{thm:formula-Tn}.

\begin{theorem} \label{thm:formula-Tn-d-children}
Let $T_n$ be the matrix presentation of the linear TCA $T_f$ defined on $\mathcal{T}_n$ with periodic boundary condition. Then
\begin{equation}
\det T_n = (-1)^{d_{n-1} - n + 1} (b + c) b^{d_{n-1} - n + 1} \sum_{j=0}^{n-1} (-1)^j c^{n-1-j} b^j,
\end{equation}
where $c = c_1 + \cdots + c_d$.
\end{theorem}

Similar to the proof of Theorem \ref{thm:formula-Tn}, we need several lemmas for the demonstration of Theorem \ref{thm:formula-Tn-d-children}.

\begin{lemma} \label{combined1}
For  $n\in \mathbb N$, let  $\overline{T}_{n}$ be defined as
\begin{equation}
\overline{T}_{n}(i,j)=
\left\{
\begin{array}{ll}
       -b, & j=1+d_{n-2},\ 1\leq{i}\leq{1+d_{n-3}}; \\
       c_1+c_2+...+c_{d}-b, &j=1+d_{n-2},\ 2+d_{n-3}\leq{i}\leq{1+d_{n-2}}; \\
        b, & j=i-1,\ 2\leq i \leq 1+d_{n-2};  \\
      c_k, & j=(i-1)d+k,\ 1\leq i\leq 1+d_{n-3}, 1 \leq k \leq d; \\
        0, & \hbox{otherwise.}
\end{array}
\right.
\end{equation}
Then $\det T_n=(-b)^{d^{n-1}-1}(b+c_1+c_2+...+c_d)\det\overline{T}_{n}$.
\end{lemma}
\begin{proof}
The proof is similar to the discussion of proof of Lemma \ref{combine21}, thus we only sketch the main steps herein.

Adding the $k$th column of $T_n$ to the first column recursively for $k \geq 2$ derives a new matrix $T_{n}^{(1)}$ defined as
\begin{equation}
T_n^{(1)}(i,j)=
\left\{
\begin{array}{ll}
b, & i=j,\ 2\leq i\leq 1+d_{n-1}; \\
c_k, & j=(i-1)d+k+1,\ 1\leq i\leq 1+d_{n-2}, 1 \leq k \leq d; \\
b+c_1+\cdots+c_d, & j=1,\ 1\leq i \leq 1+d_{n-1}; \\
0, & \hbox{otherwise.}
\end{array}
\right.
\end{equation}
It is seen that $\det T_{n}=\det{T_{n}^{(1)}}$. Since each entry in the first column of $T_n^{(1)}$ is identical, replacing the entry in the first column by $1$ and adding the $-1$ times of the $(d_{n-2} + 2)$th row to the other rows recursively produce a new matrix $T_n^{(2)}$ defined as
\begin{equation}
T_n^{(2)}(i,j)=
\left\{
\begin{array}{ll}
b, & j=i,\ 2\leq i\leq 1+d_{n-1}; \\
c_1, & j=(i-1)d+2,\ 1\leq i\leq 1+d_{n-2},\ i\neq 2+d_{n-3}; \\
c_k, & j=(i-1)d+k+1,\ 1\leq i\leq 1+d_{n-2}, 2 \leq k \leq d; \\
-b, & j=2+d_{n-2},\ 1\leq i\leq 1+d_{n-1},\ i\neq 2+d_{n-3},\ i\neq 2+d_{n-2}; \\
c_1-b, &i=2+d_{n-3},\ j=2+d_{n-2}; \\
1, & j=1,\ i=2+d_{n-2}; \\
0, & \hbox{otherwise.}
\end{array}
\right.
\end{equation}
It follows immediately that $\det T_{n}=\det T_{n}^{(1)}=(b+c_1+c_2+...+c_d)\det T_{n}^{(2)}$.

Furthermore, after adding the $(3+d_{n-2})$th, $(4+d_{n-2})$th, $\ldots$, $(1+d_{n-1})$th column to the $(2+d_{n-2})$th column of $T_n^{(2)}$, it can be verified without difficulty that $\det T_{n}^{(2)} = (-b)^{d^{n-1}-1} \det \overline{T}_{n}$. Evidently,
$$
\det T_{n} = (b+c_1+c_2+...+c_d)\det T_{n}^{(2)} = (-b)^{d^{n-1}-1} (b+c_1+c_2+...+c_d) \det \overline{T}_{n}.
$$
This completes the proof.
\end{proof}


\begin{lemma}\label{lem:alpha-n-general-case}
Let $\alpha_3=c_1+c_2+...+c_d-b$ and let $\alpha_n$ be defined recursively as
$$
\alpha_n = -\dfrac{\alpha_{n-1}(c_1+c_2+...+c_d)+b^2}{b} \quad \text{for} \quad n \geq 4.
$$
Then
$$
\alpha_n=(-1)^{n-3}\dfrac{\sum^{n-2}_{j=0}(-1)^{j}(c_1+c_2+...+c_d)^{n-2-j}(b)^{j}}{b^{n-3}}.
$$
\end{lemma}
\begin{proof}
The proof is analogous to the proof of Lemma \ref{alpha2}, hence is omitted.
\end{proof}


\begin{lemma} \label{combined2}
For $0 \leq k \leq n-3$, let
\begin{equation}
F_n^{k}(i,j)=
\left\{
\begin{array}{ll}
\alpha_{k+3}, & i=d^{n-k-1}-1, j \geq {d^{n-k-2}}; \\
\overline{T}_{n-k} (i, j), & \hbox{otherwise.}
\end{array}
\right.
\end{equation}
Then $\det \overline{T}_{n}=(-b)^{{\sum^{n-2}_{i=2}}d^i} \det F^{n-3}_{n}$.
\end{lemma}
\begin{proof}
Notably, $F_n^0 = \overline{T}_n$. Analogous to the elaboration of the proof of Lemma \ref{combine22}, observe that
\begin{enumerate}[\bf (i)]
\item adding $\dfrac{-c_d}{b}$ times of the $(1+d_{n-2})$th row, $\dfrac{-c_{d-1}}{b}$ times of the $d_{n-2}$th row, $\ldots$, $\dfrac{-c_1}{b}$ times of the $(2+d_{n-2}-d)$th row to the $(1+d_{n-3})$th row;
\item adding $\dfrac{-c_d}{b}$ times of the $(1+d_{n-2}-d)$th row, $\dfrac{-c_{d-1}}{b}$ times of the $(d_{n-2}-d)$th row, $\ldots$, $\dfrac{-c_1}{b}$ times of the $(2+d_{n-2}-2d)$th row to the $d_{n-3}$th row;
\item repeating similar operations above until adding $\dfrac{-c_d}{b}$ times of the $(1+d_{n-2}-d(d-1))$th row, $\dfrac{-c_{d-1}}{b}$ times of the $(d_{n-2}-d(d-1))$th row, $\ldots$, $\dfrac{-c_1}{b}$ times of the $(2+d_{n-2}-d^2)$th row to the $(2+d_{n-3}-d)$th row;
\end{enumerate}
it follows immediately that $\det \overline{T}_{n} = (-b)^{d^{n-2}} \det F^{1}_{n}$. Continuing similar steps recursively $(n-3)$ times derive that $\det\overline{T}_{n}=(-b)^{\sum^{n-2}_{i=2}d^i}\det F^{n-3}_{n}$. This completes the proof.
\end{proof}


\begin{proof}[Proof of Theorem \ref{thm:formula-Tn-d-children}]
It is seen that
\begin{align*}
\det{F^{n-3}_{n}}
&=\det
\left(
\begin{array}{ccccc}
c_1 & c_2 & \ldots & c_d & -b \\
b   & 0   &    &0 & \alpha_n \\
0   & b   &    &\vdots & \alpha_n  \\
\vdots  & \vdots   &  \ddots  &0 & \vdots  \\
0   & 0   &    & b& \alpha_n \\
\end{array}
\right)\\
&=(-b)^{d-1}\det
\left(
\begin{array}{cc}
c_1 & -\dfrac{(c_2+c_3+...+c_d) \alpha_n + b^2}{b} \\
b   & \alpha_n
\end{array}
\right)\\
&=(-b)^{d-1}(c \alpha_n+b^2)\\
&=(-b)^{d-1}\left[(-1)^{n-3}\dfrac{\sum^{n-2}_{j=0}(-1)^{j}c^{n-1-j} b^{j}}{b^{n-3}}+b^2\right]\\
&=(-b)^{d-1}\left[(-1)^{n-3}\dfrac{\sum^{n-2}_{j=0}(-1)^{j} c^{n-1-j}b^{j} + (-1)^{n-1} b^{n-1}}{b^{n-3}}\right]\\
&=(-b)^{d+2-n} \sum^{n-1}_{j=0}(-1)^{j}c^{n-1-j} b^{j}
\end{align*}
where $c = c_1 + c_2 + \cdots + c_d$.

Combining Lemmas \ref{combined1}, \ref{lem:alpha-n-general-case}, and \ref{combined2} infers that
\begin{align*}
\det T_n&=(-b)^{d^{n-1}-1}(b+c)\det\overline{T}_n\\
&=(b+c)(-b)^{d^{n-1}-1}(-b)^{{\sum^{n-2}_{i=2}}d^i}\det F^{n-3}_{n}\\
&=(b+c)(-b)^{(\sum^{n-1}_{i=2}d^{i})-1}(-1)^{n-3}(-b)^{d-1}\left[\sum^{n-1}_{j=0}(-1)^{j}c^{n-1-j}(b)^{j}\right]b^{3-n}\\
&=(-1)^{n-3}(b+c)(-b)^{(\sum^{n-1}_{i=1}d^{i})-2}b^{3-n} \left[\sum^{n-1}_{j=0} (-1)^{j} c^{n-1-j}b^{j}\right]\\
&=(-1)^{n-3}(b+c)(-b)^{(\sum^{n-1}_{i=1}d^{i})-2}b^{3-n} \left[\sum^{n-1}_{j=0} (-1)^{j}c^{n-1-j}b^{j}\right]\\
&=(-1)^{(\sum^{n-1}_{i=1}d^{i})+1-n}(b+c)b^{(\sum^{n-1}_{i=1}d^{i})+1-n} \left[\sum^{n-1}_{j=0} (-1)^{j}c^{n-1-j}b^{j}\right]
\end{align*}
The desired result then follows.
\end{proof}

From those cases we studied in Section 3, it is seen that the reversibility of $T_f$ depends on not only the parameters in the local rule but also $n$ and $m$. The reversibility for the general cases is much more complicated. Hence, we skip the case study for the compactness of this paper.

\section{Conclusion and Discussion}

This elucidation demonstrates that the reversibility problem of linear cellular automata defined on Cayley trees with periodic boundary condition relates to solving a polynomial derived from a recurrence relation and the coefficients of the local rules; as an example, the complete criteria of the reversibility of cellular automata over $\mathbb{Z}_2$, $\mathbb{Z}_3$, and some other specific case are addressed. Note that the grid of $\mathbb{Z}^d$ is a Cayley graph of free abelian group with $d$ generators; this makes the present study a possible approach for determining the reversibility of multidimensional cellular automata, which is known as a undecidable problem. The related exploration remains interesting.

\bibliographystyle{amsplain}
\bibliography{../../grece}

\end{document}